\documentclass{scrartcl}

\usepackage{amsmath, amsthm, amsfonts, amssymb, mathrsfs, here}
\usepackage[nocompress]{cite}
\usepackage{eucal, caption, xparse, enumitem, etoolbox, mathtools, braket, etoolbox}
\usepackage{graphicx}
\usepackage{oshs, oshs.arxiv}

\NewDocumentCommand\mres{o}{r\IfValueT{#1}{{_#1}}}
\NewDocumentCommand\length{m r()}{\ccl_{#1}\left(#2\right)}
\newtheorem{theorem}{Theorem}[section] % 1st argument is your name for it
\newtheorem{lemma}[theorem]{Lemma}     % 2nd argument is what is printed
\newtheorem{corollary}[theorem]{Corollary}

\newtheorem{conjecture}[theorem]{Conjecture}
\theoremstyle{definition}
\newtheorem{example}[theorem]{Example}
\newtheorem{remark}[theorem]{Remark}
\renewcommand\footnotemark{}
\title{On the stable Harbourne conjecture \\ for ideals defining space monomial curves} % This is the full title of the paper
\author{Kosuke Fukumuro and Yuki Irie\thanks{The second author is the corresponding author and was partially supported by JSPS KAKENHI Grant Number JP20K14277.}}
\date{}

\begin{document}

\maketitle

\begin{abstract}
For the ideal $\clp$ in $\cck[\ccx, \ccy, \ccz]$ defining a space monomial curve, we show that
$\clp^{(2 \ccn - 1)} \subseteq \clm \clp^{\ccn}$ for some positive integer $\ccn$,
where $\clm$ is the maximal ideal $(\ccx, \ccy, \ccz)$. Moreover, the smallest such $\ccn$ is determined.
It turns out that there is a counterexample to a claim due to Grifo, Huneke, and Mukundan, which states that
$\clp^{(3)} \subseteq \clm \clp^2$ if $\cck$ is a field of characteristic not $3$;
however, the stable Harbourne conjecture holds for space monomial curves as they claimed.
\end{abstract}

\section{Introduction}
\label{sec:org7d983f2}
Grifo \cite{grifo-stable-2020} proposed the following stable version of the Harbourne conjecture \cite{bauer-primer-2009, harbourne-are-2013}.

\begin{conjecture}[The stable Harbourne conjecture \cite{grifo-stable-2020}]
 \comment{Conjecture. [The stable Harbourne conjecture \cite{grifo-stable-2020}]}
\label{sec:org48c86d5}
If \(\ccI\) is a radical ideal of big height \(\ccc\) in a regular ring, then
\begin{equation}
\label{org6e11cdb}
 \ccI^{(\ccc \ccn - \ccc + 1)} \subseteq \ccI^{\ccn} \tfor \ccn \gg 0,
\end{equation}
where \(\ccI^{(\ccn)}\) is the \(\ccn\)th symbolic power of \(\ccI\).
 
\end{conjecture}

\comment{connect}
\label{sec:orgcb6a6d9}
The next theorem gives a sufficient condition to satisfy the stable Harbourne conjecture.

\begin{theorem}[Grifo, Huneke, and Mukundan \cite{grifo-expected-2020}]
 \comment{Thm. [Grifo, Huneke, and Mukundan \cite{grifo-expected-2020}]}
\label{sec:org172e0ab}
\label{org34085be}
Let \(\ccI\) be a quasi-homogeneous radical ideal of big height \(\ccc\) in a polynomial ring over a field with
irrelevant maximal ideal \(\clm\). If
\begin{equation}
\label{equ:thm-ghm}
  \ccI^{(\ccc \ccn - \ccc + 1)} \subseteq \clm \ccI^{\ccn} \tforsome \ccn,
\end{equation}
and \(\ccI^{(\ccl)}_{\clp} = \ccI^l_{\clp}\) for all prime ideals \(\clp\) with \(\clp \neq \clm\) and for all positive integers \(\ccl\), then
\[
 \sup \Set{\frac{\ccm}{\ccs} : \ccI^{(\ccm)} \not \subseteq \ccI^\ccs} < \ccc;
\]
in particular, \(\ccI\) satisfies the stable Harbourne conjecture.
 
\end{theorem}

\comment{connect}
\label{sec:orga057243}
The purpose of this paper is to prove \eqref{equ:thm-ghm} for space monomial curves.\footnote{In \cite{grifo-expected-2020}, the stable Harbourne conjecture was claimed to be true for the ideal \(\clp\) defining a space monomial curve; in the proof, the following claim was assumed: \(\clp^{(3)} \subseteq \clm \clp^2\). However, as we will see in Example \ref{orgd5c7733}, there is a counterexample to this claim.}
Let \(\cck\) be a field.
For positive integers \(n_1\), \(n_2\), and \(n_3\),
consider the ring homomorphism from \(\cck[\ccx, \ccy, \ccz]\) to \(k[t]\) defined by
\(x \mapsto t^{n_1}\), \(y \mapsto t^{n_2}\), and \(z \mapsto t^{n_3}\), where \(\cck[\ccx, \ccy, \ccz]\) and \(k[t]\) are polynomial rings over \(\cck\).
The kernel of this map is called the \emph{ideal defining the space monomial curve} \(k[t^{n_1}, t^{n_2}, t^{n_3}]\),
and is denoted by \(\clp(n_1, n_2, n_3)\). Let \(\clp = \clp(\ccn_1, \ccn_2, \ccn_3)\).
If \(\clp\) is a complete intersection, then \(\clp^{(n)} = \clp^n\) for all \(\ccn\).
Throughout this paper, \(\clp\) is assumed to be not a complete intersection.
Since the (big) height of \(\clp\) equals 2, our goal is to prove that
\begin{equation}
\label{equ:ghm-condition}
 \clp^{(2 \ccn - 1)} \subseteq \clm \clp^{\ccn} \tforsome \ccn,
\end{equation}
where \(\clm\) is the maximal ideal \((\ccx, \ccy, \ccz)\).
Note that it follows from a theorem due to Huneke and Hochster~\cite{hochster-comparison-2002} that \(\clp^{(2 \ccn)} \subseteq \clp^{\ccn}\) for all \(\ccn\).
To prove \eqref{equ:ghm-condition}, we first characterize \(\clp\) satisfying \eqref{equ:ghm-condition} for \(\ccn = 2\) (Theorem~\ref{orgeb6e5b6}).
When \(\clp^{(3)} \not \subseteq \clm \clp^2\), we show \eqref{equ:ghm-condition}
by calculating some higher symbolic powers of \(\clp\) (Theorem~\ref{org319617b}).

To state our main results precisely, we introduce some notation.
Let \(\ccR = \cck[\ccx, \ccy, \ccz]\) and 
\[
M =\begin{bmatrix}
 x^{\cca_1} & y^{\ccb_1} & z^{\ccc_1} \\
 z^{\ccc_2} & x^{\cca_2} & y^{\ccb_2} \\
\end{bmatrix},
\]
where \(a_1\), \(a_2\), \(b_1\), \(b_2\), \(c_1\), and \(c_2\) are positive integers.
Let \(\clp_{\ccM}\) or \(\ccI_2(\ccM)\) denote the ideal in \(\ccR\) generated by the maximal minors of \(\ccM\);
we write \(\clp\) instead of \(\clp_{\ccM}\) when no confusion can arise.
Then \(\idealA = (\ccF, \ccG, \ccH)\), where
\begin{align*}
 \ccF &= y^{\ccb_1 + \ccb_2} - x^{\cca_2} z^{\ccc_1},\\
 \ccG &= z^{\ccc_1 + \ccc_2} - y^{\ccb_2} x^{\cca_1},\\
 \ccH &= x^{\cca_1 + \cca_2} - z^{\ccc_2} y^{\ccb_1}.
\end{align*}
Herzog \cite{herzog-generators-1970} showed that \(\clp(n_1, n_2, n_3) = \ccI_2(\ccM)\) for some \(2 \times 3\) matrix \(\ccM\).
Without loss of generality, we may assume that one of the following two conditions holds.
\begin{enumerate}
\item \(\cca_1 \le \cca_2\), \(\ccb_1 \ge \ccb_2\), and \(\ccc_1 \ge \ccc_2\).
\item \(\cca_1 > \cca_2\), \(\ccb_1 > \ccb_2\), and \(\ccc_1 > \ccc_2\).
\end{enumerate}
Suppose that \(\ccI_2(\ccM)\) is prime.
Then \(\ccM\) is said to be of \emph{type \(1\)} and \emph{type \(2\)} if it satisfies (1) and (2), respectively.

\begin{theorem}
 \comment{Thm.}
\label{sec:org7f7a0b2}
\label{orgeb6e5b6}
Let \(\ccM\) be a matrix of type \(1\) or type \(2\).
Then \(\idealA^{(3)} \not \subseteq \clm \idealA^2\) if and only if \(\alpha_3 = \gamma_3 =  b_2 - b_1 = 0\), 
where \(\alpha_3 = \max \set{0,  2 \cca_1 - \cca_2}\) and \(\gamma_3 = \max \set{0, 2 \ccc_2 - \ccc_1}\).
 
\end{theorem}

\begin{example}
 \comment{Exm.}
\label{sec:org8b402b3}
\label{orgd5c7733}
Let \(\clp = \clp(5, 11, 4)\).
We see that \(\clp = \ccI_2(\ccM)\), where
\[
\ccM =\begin{bmatrix}
 x & y & z^3 \\
 z & x^2 & y \\
\end{bmatrix}.
\]
Since \(\alpha_3 = \gamma_3 = b_2 - b_1 = 0\),
it follows from Theorem~\ref{orgeb6e5b6} that \(\clp^{(3)} \not \subseteq \clm \clp^2\).
Indeed, let \(\ccD_2 = \ccz^{-1} (\ccH \ccF - \ccx \ccG^2)\) and \(\ccD_3 = \ccz^{-1} (\ccH \ccD_2 - \ccG^3)\).
Then
\begin{align*}
 \ccD_2 &= \ccz^{-1} \bigl((\ccx^3 - \ccy \ccz) (\ccy^2 - \ccx^2 \ccz^3) - \ccx (\ccz^4 - \ccx \ccy)^2 \bigr) \\
 & = \ccz^{-1} \bigl(- \ccx^5 \ccz^3 - \ccy^3 \ccz + \ccx^2 \ccy \ccz^4 - \ccx \ccz^8 + 2 \ccx^2 \ccy \ccz^4 \bigr) \\
 &= - \ccy^3 - \ccx^5 \ccz^2 + 3 \ccx^2 \ccy \ccz^3 - \ccx \ccz^7, \\
 \ccD_3 &= \ccz^{-1} \bigl((\ccx^3 - \ccy \ccz)(- \ccy^3 - \ccx^5 \ccz^2 + 3 \ccx^2 \ccy \ccz^3 - \ccx \ccz^7) - (\ccz^4 - \ccx \ccy)^3 \bigr) \\
 &=\ccz^{-1} \bigl (- \ccx^8 \ccz^2  + 4 \ccx^5 \ccy \ccz^3 - \ccx^4 \ccz^7 + \ccy^4 \ccz - 6 \ccx^2 \ccy^2 \ccz^4 + 4 \ccx \ccy \ccz^8 -  \ccz^{12} \bigr) \\
 &=  \ccy^4 - \ccx^8 \ccz + 4 \ccx^5 \ccy \ccz^2 - 6 \ccx^2 \ccy^2 \ccz^3 - \ccx^4 \ccz^6  + 4 \ccx \ccy \ccz^7 -  \ccz^{11}.
\end{align*}
Therefore \(\ccD_2 \in \clp^{(2)}\) and \(\ccD_3 \in \clp^{(3)}\).
Furthermore,
\begin{align*}
 \ccD_3 &= \ccy^4 - \ccx^8 \ccz + 4 \ccx^5 \ccy \ccz^2 - 6 \ccx^2 \ccy^2 \ccz^3 - \ccx^4 \ccz^6  + 4 \ccx \ccy \ccz^7 -  \ccz^{11} \\
 &= (\ccy^2 - \ccx^2 \ccz^3)^2 - \ccz^3 (\ccz^4 - \ccx \ccy)^2 - \ccx^2 \ccz (\ccx^3 - \ccy \ccz)^2  - 2 \ccx \ccz^2 (\ccz^4 - \ccx \ccy) (\ccx^3 - \ccy \ccz)\\
 &= \ccF^2 - \ccz^3 \ccG^2 - \ccx^2 \ccz \ccH^2 - 2 \ccx \ccz^2 \ccG \ccH.
\end{align*}
Hence \(\ccD_3 \not \in \clm \clp^2\), so \(\clp^{(3)} \not \subseteq \clm \clp^2\).
This is a counterexample to a claim due to Grifo, Huneke, and Mukundan~\cite[Theorem 4.2]{grifo-expected-2020},
which states that \(\clp^{(3)} \subseteq \clm \clp^2\) when the characteristic of \(\cck\) is not \(3\).
We can generalize this example as follows. Let \(\clp = \clp(\ccs + 1, \ccs^2 - \ccs - 1, \ccs)\). Then \(\clp = \ccI_2(\ccM)\), where
\[
\ccM =\begin{bmatrix}
 x & y & z^{\ccs  - 1} \\
 z & x^{\ccs  - 2} & y \\
\end{bmatrix}.
\]
Since \(\alpha_3 = \max \set{0, 2 - (\ccs - 2)}\) and \(\gamma_3 = \max \set{0, 2 - (\ccs - 1)}\),
it follows from Theorem~\ref{orgeb6e5b6} that \(\clp^{(3)} \not \subseteq \clm \clp^2\) for \(\ccs \ge 4\).
 
\end{example}

\comment{connect}
\label{sec:orgd2900d2}
By Theorem~\ref{orgeb6e5b6}, we need only show \eqref{equ:ghm-condition} for \(\ccM\) with \(\alpha_3 = \beta_3 = \ccb_1 - \ccb_2 = 0\).
Note that such \(\ccM\) is of type \(1\).

Let \(\ccM\) be a matrix of type \(1\). For a positive integer \(\ccn\), let
\begin{equation}
\label{org643446c}
 \ccD_{\ccn} = \ccx^{\alpha_{\ccn}} \ccz^{\gamma_{\ccn} - \ccc_2} (\ccH  \ccD_{\ccn - 1} - \ccx^{\cca_2 - (\ccn - 1) \cca_1} \ccy^{(\ccn - 1)(\ccb_1 - \ccb_2)} \ccG^{\ccn}),
\end{equation}
where \(\ccD_0 = - \ccy^{\ccb_2}\) and
\[
 \alpha_{\ccn} = \max \set{0, (\ccn - 1) \cca_1 - \cca_2},\ \gamma_{\ccn} = \max \set{0, (\ccn - 1) \ccc_2 - \ccc_1}.
\]
Since \(\ccz^{\ccc_2} \ccF = - \ccy^{\ccb_2} \ccH - \ccx^{\cca_2} \ccG\), we see that \(\ccD_1 = \ccF\).
Suppose that \(\ccM\) is a matrix of type \(1\) satisfying \(\ccb_1 = \ccb_2\).
Without loss of generality, we may assume that
\begin{equation}
\label{equ:type-1-p}
  \frac{\cca_2}{\cca_1} \le  \frac{\ccc_1}{\ccc_2}.
\end{equation}
If a matrix \(\ccM\) of type \(1\) satisfies \(\ccb_1 = \ccb_2\) and \eqref{equ:type-1-p},
then it is said to be of \emph{type \(1'\)}. 
Define
\[
 \mres[\ccM] = \left \lfloor \frac{\cca_2 }{\cca_1 } \right \rfloor + 1,
\]
where \(\lfloor \cca_2 / \cca_1 \rfloor\) is the greatest integer less than or equal to \(\cca_2 / \cca_1\).
Note that 
\begin{equation}
\label{equ:inequality-r}
 (\mres[\ccM] - 1) \cca_1 \le \cca_2 < \mres[\ccM] \cca_1.
\end{equation}

\begin{theorem}
 \comment{Thm.}
\label{sec:org30662f7}
\label{org319617b}
If \(\ccM\) is a matrix of type \(1'\) and \(\mres = \mres[\ccM]\), then

\begin{enumerate}
\item \(\idealA^{(\ccl)} = \idealA^{\ccl} + (\ccD_{\ccl})\) for \(0 \le \ccl \le \mres\);
\item \(\idealA^{(\mres + 1)} = \idealA^{(\mres)} \idealA + (\ccD_{\mres + 1})\);
\item \(\idealA^{(\mres + 2)} = \idealA^{(\mres + 1)} \idealA + \idealA^{(\mres)} \idealA^{(2)}\).
\end{enumerate}
 
\end{theorem}

\begin{theorem}
 \comment{Thm.}
\label{sec:orgb504ad5}
\label{org204a6ab}
If \(\ccM\) is a matrix of type \(1'\) and \(\ccn = \lfloor (\mres[\ccM] + 1)/2 \rfloor + 1\), then \(\idealA^{(2 \ccl - 1)} \not \subseteq \clm \idealA^{\ccl}\)
for \(1 \le \ccl \le \ccn - 1\) and
\begin{equation}
\label{equ:thm:stable-harbourne}
 \idealA^{(2 \ccn - 1)} \subseteq \clm \idealA^{\ccn}.
\end{equation}
 
\end{theorem}

\begin{example}
 \comment{Exm.}
\label{sec:org2fe3ae6}
\label{org8c11ad5}
Let \(\idealA = \clp(5, 11, 4)\).
Then \(\mres[\ccM] = \lfloor 2/1 \rfloor + 1 = 3\) and \(\ccn = \lfloor (\mres[\ccM] + 1)/2 \rfloor + 1 = 3\).
As we have seen in Example \ref{orgd5c7733},
\(\idealA^{(3)} = \idealA^{(2 \cdot 2 - 1)} \not \subseteq \clm \idealA^2\).
However, by Theorem \ref{org204a6ab}, \(\idealA^{(5)} = \idealA^{(2 \cdot 3 - 1)} \subseteq \clm \idealA^3\).
 
\end{example}

\comment{connect}
\label{sec:org12e8e0e}
From Theorems \ref{org34085be}, \ref{orgeb6e5b6}, and \ref{org204a6ab}, we obtain the following corollary.

\begin{corollary}
 \comment{Cor.}
\label{sec:orgbb5cdb3}
The stable Harbourne conjecture holds for ideals defining space monomial curves.
 
\end{corollary}

\comment{connect}
\label{sec:orgacd3ec7}
This paper is organized as follows.
In Section \ref{org6477ab4}, we recall some known results about the second and third symbolic powers of \(\idealA\), and
prove Theorem \ref{orgeb6e5b6}.
In Section \ref{org75e5b10}, we calculate up to the \(\mres[\ccM] + 2\)nd symbolic power of \(\idealA\) when \(\ccM\) is of type \(1'\), and show Theorem \ref{org319617b}.
Finally, we prove Theorem \ref{org204a6ab} in Section~\ref{org87a2868}. 

\section{The second and third symbolic powers}
\label{sec:org5dd7544}
\label{org6477ab4}
In this section, we prove Theorem~\ref{orgeb6e5b6}. 

\subsection{The second symbolic power}
\label{sec:org509d62e}
\begin{lemma}[\hspace{1sp}\cite{schenzel-examples-1988}]
 \comment{Lem. [\hspace{1sp}\cite{schenzel-examples-1988}]}
\label{sec:org28ba317}
\label{orgc0004c6}
If \(\ccM\) is a matrix of type \(1\) or \(2\),
then \(\idealA^{(2)} = \idealA^2 + (\ccD_2)\), where \(\ccD_2 = \ccx^{\alpha_2} \ccz^{- \ccc_2} (\ccH \ccF - x^{\cca_2 - \cca_1} y^{\ccb1 - \ccb_2} \ccG^2)\)
and \(\alpha_2 = \max \set{0, \cca_1 - \cca_2}\).
 
\end{lemma}

\begin{lemma}
 \comment{Lem.}
\label{sec:org9ad5c7c}
\label{orgce83638}
If \(\ccM\) is a matrix of type \(1\) or \(2\),
then \(\idealA^{(2)} \subseteq \clm \idealA\).
 
\end{lemma}

\begin{proof}
 \comment{Proof.}
\label{sec:org3762f59}
Since
\begin{align*}
 \ccD_2 &= \ccz^{-\ccc_2} \ccx^{\alpha_2}( \ccF \ccH - \ccx^{\cca_2 - \cca_1} \ccy^{\ccb_1 - \ccb_2} \ccG^2) \\
 &= -\ccx^{\alpha_2} (\ccy^{2\ccb_1 + \ccb_2}  - 3 \ccx^{\cca_2} \ccy^{\ccb_1} \ccz^{\cc_1} + \ccx^{\cca_1 + 2 \cca_2} \ccz^{\ccc_1 - \ccc_2} + \ccx^{\cca_2 - \cca_1} \ccy^{\ccb_1 - \ccb_2} \ccz^{2 \ccc_1 + \ccc_2})\\
 &= - \ccx^{\alpha_2} (\ccy^{\ccb_1} \ccF + \ccx^{\cca_2 - \cca_1} \ccy^{\ccb_1 - \ccb_2} \ccz^{\ccc_1} \ccG + \ccx^{\cca_2} \ccz^{\ccc_1 - \ccc_2}\ccH),
\end{align*}
it follows from Lemma \ref{orgc0004c6} that \(\idealA^{(2)} \subseteq \clm \idealA\).
\end{proof}

\subsection{The third symbolic power of type \(1\)}
\label{sec:org161a838}
\label{org0850df7}

Let \(\ccM\) be a matrix of type 1. Then \(\cca_1 \le \cca_2\), \(\ccb_1 \ge \ccb_2\), and \(\ccc_1 \ge \ccc_2\). 
Let \(\alpha = \alpha_3 = \max \set{0,  2 \cca_1 - \cca_2}\), \(\beta = \beta_3 = \max \set{0, 2 \ccb_2 - \ccb_1}\), and \(\gamma = \gamma_3 = \max \set{0, 2 \ccc_2 - \ccc_1}\).
Let
\[
 \ccD'_3 = \ccy^{\beta} \ccz^{\gamma - \ccc_2} (\ccH \ccD_2 + \ccx^{\cca_2 - \cca_1} \ccy^{\ccb_1 - 2  \ccb_2} \ccF \ccG^2).
\]

\begin{lemma}[\hspace{1sp}\cite{goto-topics-1991, knodel-explicit-1992}]
 \comment{Lem. [\hspace{1sp}\cite{goto-topics-1991, knodel-explicit-1992}]}
\label{sec:org76be569}
\label{orgbaa3c84}
If \(\ccM\) is a matrix of type \(1\), then
\[
 \idealA^{(3)} = \idealA^{(2)} \idealA + (\ccD_3, \ccD'_3). 
\]
 
\end{lemma}

\begin{lemma}
 \comment{Lem.}
\label{sec:org8457802}
\label{org06b5254}
If \(\ccM\) is a matrix of type \(1\), then \(\ccD_3\) and \(\ccD'_3\) can be expressed in terms of \(\ccF^2\), \(\ccG^2\),
\(\ccH^2\), and \(\ccG \ccH\) as follows:
\begin{equation}
\label{equ:lem-type1-formula1}
 \begin{split}
 \ccD_3 &= \ccx^{\alpha} y^{b_1 - b_2}  z^{\gamma}  F^2 -  x^{a_2 - 2  a_1 + \alpha}  y^{2  b_1 - 2  b_2} z^{c_1 + \gamma} G^2  \\
 &- x^{a_2 + \alpha} \ccz^{c_1 - 2 c_2 + \gamma} H^2 - 2  x^{a_2 - a_1 + \alpha}  y^{b_1 - b_2}  z^{c_1 - c_2 + \gamma} G  H,
 \end{split}
\end{equation}
\begin{equation}
\label{equ:lem-type1-formula2}
 \begin{split}
 \ccD'_3 &= y^{b_1 - b_2 + \beta} z^{\gamma}  F^2  - x^{2  a_2 -  a_1}  y^{b_1 - 2  b_2 + \beta} z^{c_1 - c_2 + \gamma} G^2 \\
 &- x^{a_2} \ccy^{\beta} \ccz^{c_1 - 2 c_2 + \gamma} H^2  - 2  x^{a_2 - a_1}  y^{b_1 - b_2 + \beta}  z^{c_1 - c_2 + \gamma} G  H.
 \end{split}
\end{equation}
In particular, \(\ccD_3, \ccD'_3 \in \idealA^{2}\).
 
\end{lemma}

\begin{proof}
 \comment{Proof.}
\label{sec:orgb92e8db}
A direct calculation shows that
\begin{align*}
 \ccD_2 &= \ccz^{- \ccc_2} (\ccH \ccF - \ccx^{\cca_2 - \cca_1} \ccy^{\ccb_1 - \ccb_2} \ccG^2) \\
 &= - \ccy^{2 \ccb_1 + \ccb_2} - \ccx^{\cca_1 + 2 \cca_2} \ccz^{\ccc_1 - \ccc_2} + 3 \ccx^{\cca_2} \ccy^{\ccb_1} \ccz^{\ccc_1} - \ccx^{\cca_2 - \cca_1} \ccy^{\ccb_1 - \ccb_2} \ccz^{2 \ccc_1 + \ccc_2}
\end{align*}
and
\begin{align*}
 & \ccx^{- \alpha} \ccz^{\ccc_2 - \gamma} \ccD_3 \\
 &= \ccH \ccD_2 - \ccx^{\cca_2 - 2\cca_1} \ccy^{2  \ccb_1 - 2 \ccb_2} \ccG^3 \\
 &= - \ccx^{2 \cca_1 + 3 \cca_2} z^{\ccc_1-\ccc_2} + 4 \ccx^{\cca_1 + 2 \cca_2} \ccy^{\ccb_1} z^{\ccc_1} - \ccx^{2 \cca_2} \ccy^{\ccb_1 - \ccb_2} \ccz^{2 \ccc_1 + \ccc_2}  +  \ccy^{3 \ccb_1 + \ccb_2} \ccz^{\ccc_2}\\
 & \phantom{=} - 6 \ccx^{\cca_2} \ccy^{2 \ccb_1} \ccz^{\ccc_1 + \ccc_2} + 4 \ccx^{\cca_2 - \cca_1} \ccy^{2 \ccb_1 - \ccb_2} \ccz^{2 \ccc_1 + 2 \ccc_2} - \ccx^{\cca_2 -2 \cca_1} \ccy^{2 \ccb_1 - 2 \ccb_2} \ccz^{3 \ccc_1 + 3 \ccc_2} \\
 &= \ccz^{\ccc_2} \bigl(\ccy^{\ccb_1 - \ccb_2}  \ccF^2 -  \ccx^{\cca_2 - 2 \cca_1}  \ccy^{2 \ccb_1 - 2 \ccb_2} \ccz^{\ccc_1} \ccG^2 - \ccx^{\cca_2} \ccz^{\ccc_1 - 2 \ccc_2} \ccH^2 - 2  \ccx^{\cca_2 - \cca_1}  \ccy^{\ccb_1 - \ccb_2} \ccz^{\ccc_1 - \ccc_2} \ccG \ccH \bigr).
 \end{align*}
This yields \eqref{equ:lem-type1-formula1}. We also see that \(\ccD_3 \in \idealA^2\) since the coefficients of \(F^2\), \(\ccG^2\), \(\ccH^2\), and \(\ccG \ccH\) are in \(\cck[\ccx, \ccy, \ccz]\).
Similarly,
\begin{align*}
 & \ccy^{- \beta} \ccz^{\ccc_2 - \gamma} \ccD'_3 \\
 &= \ccH \ccD_2 + \ccx^{\cca_2 - \cca_1} \ccy^{\ccb_1 - 2  \ccb_2} \ccF \ccG^2 \\
 &= - \ccx^{2 \cca_1+3 \cca_2} \ccz^{\ccc_1-\ccc_2} + 3 \ccx^{\cca_1 + 2 \cca_2} \ccy^{\ccb_1} \ccz^{\ccc_1} + \ccx^{2 \cca_2} \ccy^{\ccb_1-\ccb_2} \ccz^{2 \ccc_1+\ccc_2} + \ccy^{3 \ccb_1+\ccb_2} \ccz^{\ccc_2}  \\
 & \phantom{=}  -5 \ccx^{\cca_2} \ccy^{2 \ccb_1} \ccz^{\ccc_1+\ccc_2} + 2 \ccx^{\cca_2 -\cca_1} \ccy^{2 \ccb_1-\ccb_2} \ccz^{2 \ccc_1+2 \ccc_2} - \ccx^{2 \cca_2 -\cca_1} \ccy^{\ccb_1-2 \ccb_2} \ccz^{3 \ccc_1+2 \ccc_2} \\
 &= \ccz^{\ccc_2} \bigl(\ccy^{\ccb_1 - \ccb_2} \ccF^2  - \ccx^{2  \cca_2 -  \cca_1}  \ccy^{\ccb_1 - 2  \ccb_2} \ccz^{\ccc_1 - \ccc_2} \ccG^2 - \ccx^{\cca_2} \ccz^{\ccc_1 - 2 \ccc_2} \ccH^2  - 2  \ccx^{\cca_2 - \cca_1}  \ccy^{\ccb_1 - \ccb_2}  \ccz^{\ccc_1 - \ccc_2} \ccG \ccH \bigr).
\end{align*}
Therefore \eqref{equ:lem-type1-formula2} holds and \(\ccD_3' \in \idealA^2\).
\end{proof}

\begin{remark}
 \comment{Rem.}
\label{sec:org5e069c9}
\label{orgbdcb8f2}
Let \(\ccM\) be a matrix of type 1.
It follows from Lemmas \ref{orgbaa3c84} and \ref{org06b5254} that \(\idealA^{(3)} \subseteq \idealA^2\).
This fact, which also holds for a matrix of type 2, 
was proved by Nishida \cite{kojinishida-third-2008} for an arbitrary field and independently by Grifo \cite{grifo-stable-2020} for a field of characteristic not \(3\).
 
\end{remark}

\begin{theorem}
 \comment{Thm.}
\label{sec:org9c64f4e}
\label{org5623dd8}
If \(\ccM\) is a matrix of type \(1\),
then \(\idealA^{(3)} \not \subseteq \clm \idealA^2\) if and only if 
\begin{equation}
\label{equ:thm-type-1}
 \alpha = \gamma =  b_2 - b_1 = 0.
\end{equation}
 
\end{theorem}

\begin{proof}
 \comment{Proof.}
\label{sec:orgced3f3b}
First, suppose that \eqref{equ:thm-type-1} holds. 
Then the coefficient of \(\ccF^2\) in \eqref{equ:lem-type1-formula1} is equal to 1.
Therefore \(\ccD_3 \not \in \clm \idealA^2\) because, among the six generators \(\ccF^2\), \(\ccG^2\), \(\ccH^2\), \(\ccF \ccG\), \(\ccF \ccH\), and \(\ccG \ccH\) of \(\idealA^2\), the polynomial \(F^2\) only contains a power of \(\ccy\) as a term. 
It follows from Lemma \ref{orgbaa3c84} that \(\idealA^{(3)} \not \subseteq \clm \idealA^2\).

Next, suppose that \eqref{equ:thm-type-1} does not hold.
Then the coefficients of \(\ccF^2\), \(\ccG^2\), \(\ccH^2\), and \(\ccG \ccH\) in
\eqref{equ:lem-type1-formula1} and \eqref{equ:lem-type1-formula2} are in \(\clm\).
Therefore \(\ccD_3, \ccD_3' \in \clm \idealA^2\).
By Lemmas \ref{orgce83638} and \ref{orgbaa3c84}, we see that \(\idealA^{(3)} \subseteq \clm \idealA^2\).
\end{proof}

\subsection{The third symbolic power of type \(2\)}
\label{sec:orgd9ef487}
\label{org0f88f4e}
Let  \(\ccM\) be a matrix of type 2. Then \(\cca_1 > \cca_2\), \(\ccb_1 > \ccb_2\), and \(\ccc_1 > \ccc_2\). 
Let \(\alpha = \max \set{0, 2 \cca_2 - \cca_1}\), \(\beta = \max \set{0, 2 \ccb_2 - \ccb_1}\), and \(\gamma = \max \set{0, 2 \ccc_2 - \ccc_1}\).
Note that the definition of \(\alpha\) is different from that of in Section \ref{org0850df7}.
Let \(\sigma\) be the permutation \((x, y, z)(a_1, b_1, c_1)(a_2, b_2, c_2)(\alpha, \beta, \gamma)\).
Note that \(\ccF^\sigma =  \ccG\) and \(\ccF^{\sigma^2} = \ccH\).
Let
\begin{align*}
 \ccD_3 &= \ccz^{- \ccc_2} \ccy^{\beta} \ccx^{\alpha}(\ccx^{\cca_1 - 2 \cca_2}  \ccH^2 \ccF + \ccy^{\ccb_1 - 2  \ccb_2} \ccG^3 ), \\
 \ccD_3' = \ccD_3^{\sigma} &= \ccx^{- \cca_2} \ccz^{\gamma} \ccy^{\beta}(\ccy^{\ccb_1 - 2 \ccb_2} \ccF^2 \ccG + \ccz^{\ccc_1 - 2  \ccc_2}  \ccH^3),\\
 \ccD_3''= \ccD_3^{\sigma^2}  &= \ccy^{- \ccb_2} \ccx^{\alpha} \ccz^{\gamma}(\ccz^{\ccc_1 - 2 \ccc_2} \ccG^2 \ccH + \ccx^{\cca_1 - 2  \cca_2}  \ccF^3 ).
\end{align*}

\begin{lemma}[\hspace{1sp}\cite{goto-topics-1991, knodel-explicit-1992}]
 \comment{Lem. [\hspace{1sp}\cite{goto-topics-1991, knodel-explicit-1992}]}
\label{sec:orgbea4507}
\label{orgd0a69b7}
Let \(\ccM\) be a matrix of type \(2\). If the characteristic of \(\cck\) is not \(2\), then
\[
 \idealA^{(3)} = \idealA^{(2)} \idealA + (\ccD_3, \ccD'_3, \ccD_3''). 
\]
If the characteristic of \(\cck\) is \(2\), then \(x^{- \cca_2 + \alpha} \ccD_3 = y^{- \ccb_2 + \beta} \ccD_3 = z^{- \ccc_2 + \gamma} \ccD_3 \in k[x,y,z]\) and
\[
 \idealA^{(3)} = \idealA^{(2)} \idealA + (\ccz^{- \ccc_2 + \gamma} \ccD_3). 
\]
 
\end{lemma}

\begin{lemma}
 \comment{Lem.}
\label{sec:org6c9adc3}
\label{org745fc89}
Let \(\ccM\) be a matrix of type \(2\). Then
\begin{equation}
\label{equ:lem-type2-formula}
\begin{split}
 \ccD_3 &= \ccx^{\cca_1 - 2  \cca_2 + \alpha}  \ccy^{\ccb_1 - \ccb_2 + \beta} \ccz^{\ccc_2} \ccF^2 
 + \ccx^{\alpha}  \ccy^{\ccb_1 - 2  \ccb_2 + \beta} \ccz^{\ccc_1} \ccG^2 \\
 &\phantom{=} - \ccx^{\cca_1 - \cca_2 + \alpha}  \ccy^{\beta} \ccz^{\ccc_1 - \ccc_2}  \ccH^2  
 +  2 \ccx^{\cca_1 - \cca_2 + \alpha}  \ccy^{\ccb_1 - \ccb_2 + \beta}  \ccF  \ccG
\end{split}
\end{equation}
and \(\ccD_3 \in \idealA^2\).
 
\end{lemma}

\begin{proof}
 \comment{Proof.}
\label{sec:orga9e6398}
A direct calculation shows that
\begin{align*}
 &\ccx^{-\alpha} \ccy^{-\beta} \ccz^{\ccc_2} \ccD_3 \\
 &= \ccx^{\cca_1 - 2  \cca_2}  \ccH^2  \ccF + \ccy^{\ccb_1 - 2  \ccb_2}  \ccG^3  \\
 &= - 2 \ccx^{2 \cca_1 - \cca_2} \ccy^{2 \ccb_1 + \ccb_2} \ccz^{\ccc_2} - \ccx^{3 \cca_1 + \cca_2} \ccz^{\ccc_1} + 5 \ccx^{2 \cca_1} \ccy^{\ccb_1} \ccz^{\ccc_1 + \ccc_2} + \ccx^{\cca_1 - 2 \cca_2} \ccy^{3 \ccb_1 + \ccb_2} \ccz^{2\ccc_2}\\
 & \phantom{= }-\ccx^{\cca_1 - \cca_2} \ccy^{2 \ccb_1} \ccz^{\ccc_1 + 2\ccc_2} - 3 \ccx^{\cca_1} \ccy^{\ccb_1 - \ccb_2} \ccz^{2 \ccc_1 + 2 \ccc_2} + \ccy^{\ccb_1 - 2 \ccb_2} \ccz^{3 \ccc_1 + 3 \ccc_2}  \\
 &= \ccz^{\ccc_2} \bigl (\ccx^{\cca_1 - 2  \cca_2}  \ccy^{\ccb_1 - \ccb_2} \ccz^{\ccc_2} \ccF^2 + \ccy^{\ccb_1 - 2  \ccb_2} \ccz^{\ccc_1} \ccG^2
 - \ccx^{\cca_1 - \cca_2} \ccz^{\ccc_1 - \ccc_2} \ccH^2  + 2 \ccx^{\cca_1 - \cca_2}  \ccy^{\ccb_1 - \ccb_2}  \ccF  \ccG \bigr).
\end{align*}
Therefore \eqref{equ:lem-type2-formula} holds and \(\ccD_3 \in \idealA^2\).
\end{proof}

\begin{theorem}
 \comment{Thm.}
\label{sec:org46874cb}
\label{org72d0da8}
If \(\ccM\) is a matrix of type \(2\), then \(\idealA^{(3)} \subseteq \clm \idealA^2\).
 
\end{theorem}

\begin{proof}
 \comment{Proof.}
\label{sec:org2e2cdfc}
First, suppose that the characteristic of \(\cck\) is not \(2\).
Since the coefficients of \(F^2, G^2, H^2\) and \(GH\) in \eqref{equ:lem-type2-formula} are in \(\clm\),
we see that \(\ccD_3 \in \clm \idealA^2\). By symmetry, \(\ccD'_3, \ccD''_3 \in \clm \idealA^2\). 
It follows from Lemmas \ref{orgce83638} and \ref{orgd0a69b7} that \(\idealA^{(3)} \subseteq \clm \idealA^2\).

Next, suppose that the characteristic of \(\cck\) is \(2\).
By Lemma \ref{org745fc89},
\begin{equation}
\label{equ:thm-type2}
\begin{split}
 \ccz^{- \ccc_2 + \gamma} \ccD_3 &= - \ccx^{\cca_1 - 2  \cca_2 + \alpha}  \ccy^{\ccb_1 - \ccb_2 + \beta} \ccz^{\gamma} \ccF^2 - \ccx^{\alpha}  \ccy^{\ccb_1 - 2  \ccb_2 + \beta} \ccz^{\ccc_1 - \ccc_2 + \gamma} \ccG^2\\
 &\phantom{=} + \ccx^{\cca_1 - \cca_2 + \alpha}  \ccy^{\beta} \ccz^{\ccc_1 - 2\ccc_2 + \gamma}  \ccH^2.
\end{split}
\end{equation}
Since the coefficients of \(F^2\), \(G^2\), and \(H^2\) in \eqref{equ:thm-type2} are in \(\clm\),
it follows that \(\ccz^{- \ccc_2 + \gamma} \ccD_3 \in \clm \idealA^2\).
Therefore \(\idealA^{(3)} \subseteq \clm \idealA^2\).
\end{proof}

\section{Up to the \(r + 2\)nd symbolic power of type \(1'\)}
\label{sec:orgbfd5a1f}
\label{org75e5b10}
Let \(\ccM\) be a matrix of type \(1'\). 
Then \(\cca_1 \le \cca_2\), \(\ccb_1 = \ccb_2\), \(\ccc_1 \ge \ccc_2\), and \(\cca_2/\cca_1 \le \ccc_1/\ccc_2\).
Let \(\mres = \mres[\ccM] = \lfloor \cca_2/\cca_1 \rfloor + 1\).
In this section, we calculate up to the \(\mres + 2\)nd symbolic power of \(\idealA\).
Let \(\ccb = \ccb_1\) and \(\myy = \ccy^{\ccb}\).
Note that
\[
 \length{\ccR}(\ccR/\idealA + (\ccx)) = (\ccc_1 + 2 \ccc_2) \ccb \tand \length{\ccR}(\ccR/\idealA + (\ccz)) = (2 \cca_1 + \cca_2) \ccb,
\]
where \(\ccR = \cck[\ccx, \ccy, \ccz]\) and \(\length{\ccR}(\ccR/\idealA + (\ccx))\) is the length of the \(\ccR\)-module \(\ccR/\idealA + (\ccx)\).
By \cite[Lemma 2.3]{schenzel-examples-1988}, we obtain the following result.

\begin{lemma}
 \comment{Lem.}
\label{sec:org98dc84c}
\label{org9fefb4d}
Let \(\ccM\) be a matrix of type \(1'\).
If \(\idealB \subseteq \idealA^{(\ccn)}\) and one of the following two conditions holds, then \(\idealB = \idealA^{(\ccn)}\).
\begin{enumerate}
\item \(\displaystyle \length{\ccR}(\ccR/\idealB + (\ccx)) = \binom{\ccn + 1}{2} (\ccc_1 + 2 \ccc_2) \ccb\);
\item \(\displaystyle \length{\ccR}(\ccR/\idealB + (\ccz)) = \binom{\ccn + 1}{2} (2 \cca_1 + \cca_2) \ccb\).
\end{enumerate}
 
\end{lemma}

\comment{connect}
\label{sec:org19ba373}
The next lemma enables us to calculate the lengths of certain modules.

\begin{lemma}[\hspace{1sp}\cite{goto-gorensteinness-1993}]
 \comment{Lem. [\hspace{1sp}\cite{goto-gorensteinness-1993}]}
\label{sec:orgc8cd029}
\label{orgcc5bc69}
Let \(\ccu_0, \dotsc, \ccu_{\ccn}, \ccv_0, \dotsc, \ccv_{\ccn}\) be non-negative integers such that
\[
 \ccu_0 \ge \ccu_1 \ge \dotsb \ge \ccu_{\ccn - 1} > \ccu_{\ccn} = 0 \tand \ccv_0 = 0 < \ccv_1 \le \dotsb \le \ccv_{\ccn}.
\]
If \(\ccR = \cck[\ccx, \ccy]\) and
 \(\ccI = (\ccx^{\ccu_\cci} \ccy^{\ccv_\ccj} : \cci + \ccj = \ccn)\), then
\[
 \length{\ccR}(\ccR/\ccI) = \sum_{\cci = 0}^{\ccn - 1} \ccu_\cci (\ccv_{\cci + 1} - \ccv_{\cci}).
\]
 
\end{lemma}

\comment{connect}
\label{sec:org6d2f579}
To apply Lemma \ref{orgcc5bc69}, we calculate \(\ccD_{\ccl} + (\ccx)\) and  \(\ccD_{\ccl} + (\ccz)\) in the next two lemmas.

\begin{lemma}
 \comment{Lem.}
\label{sec:org89bb523}
\label{org2a0af15}
Let \(\ccM\) be a matrix of type \(1'\) and \(\mres = \mres[\ccM]\).
If \(1 \le \ccl \le \mres\), then \(\ccD_{\ccl} \in \idealA^{(\ccl)}\) and
\begin{equation}
\label{equ:lem-dn-1}
 \ccD_{\ccl} \equiv  ( - \myy)^{\ccl + 1} - x^{(\ccl - 1)\cca_1 + \ccl \cca_2} \ccz^{\ccc_1 - (\ccl - 1) \ccc_2} \pmod{\ccx^{\cca_2 - (\ccl - 1) \cca_1} \ccz^{\ccc_1 - (\ccl - 1) \ccc_2 + 1}};
\end{equation}
moreover, 
\[
\ccD_{\ccl} \equiv ( - \myy)^{\ccl + 1} \begin{cases}
 \! \! \! \! \pmod{\ccx} & \tif \cca_2 > (\mres - 1) \cca_1 \\
 \! \! \! \! \pmod{\ccz} & \tif \cca_2 = (\mres - 1) \cca_1.
 \end{cases}
\]
 
\end{lemma}

\begin{proof}
 \comment{Proof.}
\label{sec:orgd892e4f}
We show \eqref{equ:lem-dn-1} by induction on \(\ccl\).
If \(\ccl = 1\), then it is obvious.
Suppose that \(\ccl \ge 2\). 
By the induction hypothesis,
\[
 \ccD_{\ccl - 1} = (- \myy)^{\ccl} - x^{(\ccl - 2)\cca_1 + (\ccl - 1) \cca_2} \ccz^{\ccc_1 - (\ccl - 2) \ccc_2} + \ccA \ccx^{\ccN} \ccz^{\ccM} \tforsome \ccA \in \cck[\ccx, \ccy, \ccz], 
\]
where \(\ccN = \cca_2 - (\ccl - 2) \cca_1\) and \(\ccM = \ccc_1 - (\ccl - 2) \ccc_2 + 1\).
We also see that
\[
 \ccG^{\ccl} = (- \myy \ccx^{\cca_1} + \ccz^{\ccc_1 + \ccc_2})^{\ccl} = (- \myy)^{\ccl} \ccx^{\ccl \cca_1} + \ccB \ccz^{\ccc_1 + \ccc_2} \tforsome \ccB \in \cck[\ccx, \ccy, \ccz].
\]
Therefore
\begin{align*}
 \ccz^{\ccc_2} \ccD_{\ccl} &=  \ccH \ccD_{\ccl - 1} - \ccx^{\cca_2 - (\ccl - 1) \cca_1} \ccG^{\ccl} \\
 &=  (\ccx^{\cca_1 + \cca_2} - \myy \ccz^{\ccc_2}) \bigl((- \myy)^{\ccl} - x^{(\ccl - 2)\cca_1 + (\ccl - 1) \cca_2} \ccz^{\ccc_1 - (\ccl - 2) \ccc_2} + \ccA \ccx^{\ccN} \ccz^{\ccM} \bigr) \\
 &\phantom{=} - (- \myy)^{\ccl} \ccx^{\cca_1 + \cca_2} - \ccB \ccx^{\cca_2 - (\ccl - 1) \cca_1} \ccz^{\ccc_1 + \ccc_2} \\
 &= - \ccx^{(\ccl - 1) \cca_1 + \ccl \cca_2} \ccz^{\ccc_1 - (\ccl - 2) \ccc_2} + \ccA \ccx^{\cca_1 + \cca_2 + \ccN} \ccz^{\ccM} \\
 &\phantom{=}+ (- \myy)^{\ccl + 1} \ccz^{\ccc_2} + \myy \ccx^{(\ccl - 2) \cca_1 + (\ccl - 1) \cca_2} \ccz^{\ccc_1 - (\ccl - 3) \ccc_2} - \ccA \myy \ccx^{\ccN} \ccz^{\ccM + \ccc_2} \\
 &\phantom{=} - \ccB \ccx^{\cca_1 + (\ccl - 1)  \cca_2} \ccz^{\ccc_1 +  \ccc_2},
\end{align*}
so 
\begin{align*}
 \ccD_{\ccl} &= - \ccx^{(\ccl - 1) \cca_1 + \ccl \cca_2} \ccz^{\ccc_1 - (\ccl - 1) \ccc_2} + \ccA \ccx^{\cca_1 + \cca_2 + \ccN} \ccz^{\ccM - \ccc_2} \\
 &\phantom{=}+ (- \myy)^{\ccl + 1} + \myy \ccx^{(\ccl - 2) \cca_1 + (\ccl - 1) \cca_2} \ccz^{\ccc_1 - (\ccl - 2) \ccc_2} - \ccA \myy \ccx^{\ccN} \ccz^{\ccM} \\
 &\phantom{=} - \ccB \ccx^{\cca_1 + (\ccl - 1)  \cca_2} \ccz^{\ccc_1} \\
 &\equiv ( - \myy)^{\ccl + 1} - \ccx^{(\ccl - 1)\cca_1 + \ccl \cca_2} \ccz^{\ccc_1 - (\ccl - 1) \ccc_2} \pmod{\ccx^{\cca_2 - (\ccl - 1) \cca_1} \ccz^{\ccc_1 - (\ccl - 1) \ccc_2 + 1}}
\end{align*}
since \(\ccM - \ccc_2 = \ccc_1 - (\ccl - 1) \ccc_2 + 1\).
Thus \(\ccD_{\ccl} \in \idealA^{(\ccl)}\) and \eqref{equ:lem-dn-1} holds.
If \(\cca_2 > (\mres - 1) \cca_1\), then \(\cca_2 - (\ccl - 1) \cca_1 \ge \cca_2 - (\mres - 1) \cca_1 \ge 1\),
so \(\ccD_{\ccl} \equiv ( - \myy)^{\ccl + 1} \pmod{\ccx}\).
Suppose that \(\cca_2 = (\mres - 1) \cca_1\).
Recall that \(\ccc_1/\ccc_2 \ge \cca_2/\cca_1 = \mres - 1\).
We show that \(\ccc_1 > (\mres - 1) \ccc_2\).
Assume that \(\ccc_1 = (\mres - 1) \ccc_2\).
Let \(\ccX = \ccx^{\cca_1}\) and \(\ccZ = \ccz^{\ccc_2}\).
Then \(\idealA = (\ccX^{\mres} - \myy \ccZ, \ccZ^{\mres} - \myy \ccX, \myy^2 - \ccX^{\mres - 1} \ccZ^{\mres - 1})\).
Since
\[
 \ccX^{\mres} - \myy \ccZ - (\ccZ^{\mres} - \myy \ccX) = (\ccX - \ccZ)(\ccX^{\mres - 1} + \ccX^{\mres - 2} \ccZ + \cdots + \ccX \ccZ^{\mres - 2} + \ccZ^{\mres - 1} + \myy),
\]
it follows that \(\idealA\) is not a prime ideal, which is a contradiction. Hence \(\ccc_1 > (\mres - 1) \ccc_2\).
Therefore, for \(1 \le \ccl \le \mres\), we see that \(\ccc_1 - (\ccl - 1) \ccc_2 \ge \ccc_1 - (\mres - 1) \ccc_2 \ge 1\),
so \(\ccD_{\ccl} \equiv ( - \myy)^{\ccl + 1} \pmod{\ccz}\).
\end{proof}

\begin{example}
 \comment{Exm.}
\label{sec:org2637241}
\label{orgf829687}
Let \(\ccM\) be as in Example \ref{orgd5c7733}. Then \(\mres[\ccM] = 3\) and \(\cca_2 = 2 = 2 \cdot 1 = (\mres[\ccM] - 1) \cca_1\).
Since
\[
 \ccD_3 = \ccy^4 - \ccx^8 \ccz + 4 \ccx^5 \ccy \ccz^2 - 6 \ccx^2 \ccy^2 \ccz^3 - \ccx^4 \ccz^6  + 4 \ccx \ccy \ccz^7 -  \ccz^{11},
\]
it follows that
\[
 \ccD_3 \equiv \begin{cases}
 \ccy^4 - \ccz^{11}  & \pmod{\ccx}, \\
 \ccy^4 & \pmod{\ccz}.
 \end{cases}
\]
 
\end{example}

\begin{lemma}
 \comment{Lem.}
\label{sec:org6b7777c}
\label{org5b49321}
Let \(\ccM\) be a matrix of type \(1'\), \(\mres = \mres[\ccM]\), and \(\gamma = \gamma_{\mres + 1}\). Then \(\ccD_{\mres + 1} \in \idealA^{(\mres + 1)}\) and
\begin{equation}
\label{equ:lem-dn-2}
 \ccD_{\mres + 1} \equiv - \ccz^{(\mres + 1) \ccc_1 + \mres \ccc_2 + \gamma} \pmod{\ccx}.
\end{equation}
 
\end{lemma}

\begin{proof}
 \comment{Proof.}
\label{sec:org3e57188}
By Lemma \ref{org2a0af15},
\[
 \ccD_{\mres} = ( - \myy)^{\mres + 1} + \ccA \ccz^{\ccc_1 - (\mres - 1) \ccc_2} \tforsome \ccA \in \cck[\ccx, \ccy, \ccz]. 
\]
We also see that
\[
 \ccG^{\mres + 1} = (- \myy)^{\mres + 1} \ccx^{(\mres + 1) \cca_1} + \ccB (- \myy) x^{\cca_1} \ccz^{\ccc_1 + \ccc_2} + \ccz^{(\mres + 1) (\ccc_1 + \ccc_2)} \tforsome \ccB \in \cck[\ccx, \ccy, \ccz].
\]
Hence
\begin{align*}
 \ccz^{\ccc_2 - \gamma} \ccD_{\mres + 1} &= \ccx^{\mres \cca_1 - \cca_2} \ccH \ccD_{\mres} - \ccG^{\mres + 1} \\
 &= \ccx^{\mres \cca_1 - \cca_2} \bigl(\ccx^{\cca_1 + \cca_2} - \myy \ccz^{\ccc_2}\bigr) \bigl(( - \myy)^{\mres + 1} + \ccA \ccx^{\ccc_1 - (\mres - 1) \ccc_2} \bigr) \\
 &\phantom{=}- (- \myy)^{\mres + 1} \ccx^{(\mres + 1) \cca_1} - \ccB (- \myy) \ccx^{\cca_1} \ccz^{\ccc_1 + \ccc_2} - \ccz^{(\mres + 1) (\ccc_1 + \ccc_2)} \\
&= \ccA \ccx^{(\mres + 1) \cca_1} \ccz^{\ccc_1 - (\mres - 1) \ccc_2} + \ccx^{\mres \cca_1 - \cca_2} (-\myy) \ccz^{\ccc_2} \bigl(( - \myy)^{\mres + 1} + \ccA \ccx^{\ccc_1 - (\mres - 1) \ccc_2} \bigr) \\
 &\phantom{=}- \ccB (- \myy) \ccx^{\cca_1} \ccz^{\ccc_1 + \ccc_2} - \ccz^{(\mres + 1) (\ccc_1 + \ccc_2)}.
\end{align*}
Since \(\ccc_1 - (\mres - 1) \ccc_2 - (\ccc_2 - \gamma) = \ccc_1 - \mres \ccc_2 + \gamma \ge 0\), it follows that 
\begin{align*}
 \ccD_{\mres + 1} &= \ccA \ccx^{(\mres + 1) \cca_1} \ccz^{\ccc_1 - \mres \ccc_2 + \gamma} + \ccx^{\mres \cca_1 - \cca_2} (-\myy) \ccz^{\gamma} \bigl(( - \myy)^{\mres + 1} + \ccA \ccx^{\ccc_1 - (\mres - 1) \ccc_2} \bigr) \\
 &\phantom{=}- \ccB (- \myy) \ccx^{\cca_1} \ccz^{\ccc_1 + \gamma} - \ccz^{(\mres + 1) \ccc_1 + \mres \ccc_2 + \gamma} \in  \idealA^{(\mres + 1)}
\end{align*}
and that \eqref{equ:lem-dn-2} holds.
\end{proof}

\comment{connect Proof}
\label{sec:orgd048b50}
\vspace{0.5em}

\noindent
\emph{Proof of Theorem~\ref{org319617b}}. 
We divide the proof into two cases.

\comment{connect: case 1}
\label{sec:org968a681}
\paragraph*{Case 1 $(\mres - 1) \cca_1 < \cca_2$.}

By Lemmas \ref{org2a0af15} and \ref{org5b49321},
\begin{alignat*}{3}
 \ccD_{\ccl} &\equiv (- \ccY)^{\ccl + 1} && \pmod{\ccx} \tfor 1 \le \ccl \le \mres, \\ 
 \ccD_{\mres + 1} &\equiv -\ccz^{(\mres + 1) \ccc_1 + \mres  \ccc_2 + \gamma} && \pmod{\ccx}.
\end{alignat*}

\comment{connect: (i)}
\label{sec:orgfb3becf}
\paragraph*{\textup{(i)}}
Let 
\[
 \idealB = (\ccD_{\ccl}) + (\ccG^{\cci} \ccH^{\ccj} : \cci + \ccj = \ccl,\ 0 \le \ccj \le \ccl).
\]
Then \(\idealB \subseteq (\ccD_{\ccl}) + \idealA^{\ccl} \subseteq \idealA^{(\ccl)}\).
We show that \(\length{\ccR}(\ccR/\idealB + (\ccx)) = \binom{\ccl + 1}{2} \ccb (\ccc_1 + 2 \ccc_2)\).
Since 
\begin{alignat*}{3}
 \ccD_{\ccl} &\equiv (- \ccY)^{\ccl + 1} &&   \pmod{\ccx}, \\ 
 \ccG^{\cci} \ccH^{\ccj} &\equiv (-\myy)^\ccj \ccz^{(\ccl - \ccj) \ccc_1 + \ccl \ccc_2} && \pmod{\ccx} \tfor 0 \le \ccj \le \ccl,
\end{alignat*}
it follows from Lemma \ref{orgcc5bc69} that
\[
 \frac{\length{\ccR}(\ccR/\idealB + (\ccx))}{\ccb} = (\ccl + 1) \ccl \ccc_2 + \sum_{\ccj = 1}^{\ccl} \ccj \ccc_1 = \binom{\ccl + 1}{2} (\ccc_1 + 2  \ccc_2).
\]
Therefore \(\idealB = \idealA^{(\ccl)}\) by Lemma \ref{org9fefb4d}.

\comment{connect: (ii)}
\label{sec:org26fdf84}
\paragraph*{\textup{(ii)}}
Let
\[
 \idealB = (\ccF \ccD_{\mres}, \ccH \ccD_{\mres}, \ccG \ccD_{\mres}, \ccD_{\mres + 1}) + (\ccG^{\cci} \ccH^{\ccj} : \cci + \ccj = \mres + 1,\ 1 \le \ccj \le \mres + 1).
\]
Then \(\idealB \subseteq (\ccD_{\mres + 1}) + \idealA^{(\mres)} \idealA \subseteq \idealA^{(\mres + 1)}\).
We show that \(\length{\ccR}(\ccR/\idealB + (\ccx)) = \binom{\mres + 2}{2} \ccb (\ccc_1 + 2 \ccc_2)\).
Let \(\gamma = \gamma_{\mres + 1} = \max \set{0, \mres \ccc_2 - \ccc_1}\). Then
\begin{alignat*}{3}
 \ccF \ccD_{\mres} &\equiv (-\myy)^{\mres + 3} && \pmod{\ccx},\\
 \ccH \ccD_{\mres} &\equiv (-\myy)^{\mres + 2} z^{\ccc_2} && \pmod{\ccx},\\
 \ccG \ccD_{\mres} &\equiv (-\myy)^{\mres + 1} z^{\ccc_1 + \ccc_2} && \pmod{\ccx},\\
 \ccH^{\mres + 1} &\equiv (-\myy)^{\mres + 1} \ccz^{(\mres + 1) \ccc_2} && \pmod{\ccx},\\
 \ccG^{\cci} \ccH^{\ccj} &\equiv (-\myy)^\ccj \ccz^{(\mres + 1 - \ccj) \ccc_1 + (\mres + 1) \ccc_2} && \pmod{\ccx} \tfor 1 \le \ccj \le \mres, \\
 \ccD_{\mres + 1} &\equiv -\ccz^{(\mres + 1) \ccc_1 + \mres  \ccc_2 + \gamma} && \pmod{\ccx}.
\end{alignat*}
Note that
\[
 \min \set{\ccc_1 + \ccc_2, (\mres + 1) \ccc_2} = (\mres + 1)\ccc_2 - \gamma.
\]
It follows that
\begin{align*}
 \frac{\length{\ccR}(\ccR/\idealB + (\ccx))}{\ccb} &=  (\mres + 3) \ccc_2 + (\mres + 2) (\mres \ccc_2 - \gamma) + (\mres + 1) (\ccc_1 + \gamma)  \\
 &\phantom{=} + \Bigl(\sum_{\ccj = 2}^{\mres} \ccj \ccc_1  \Bigr) +  (\ccc_1 - \ccc_2 + \gamma) \\
 &= (\mres^2 + 3 \mres + 2) \ccc_2 + \binom{\mres + 2}{2} \ccc_1 = \binom{\mres + 2}{2}  (\ccc_1 + 2 \ccc_2).
\end{align*}
Therefore \(\idealB = \idealA^{(\mres + 1)}\).

\comment{connect: (iii)}
\label{sec:org0334679}
\paragraph*{\textup{(iii)}}
Let 
\[
 \idealB = (\ccD_2 \ccD_{\mres}, \ccH^2 \ccD_{\mres}, \ccG \ccH \ccD_{\mres}, \ccG^2 \ccD_{\mres}, \ccH \ccD_{\mres + 1}, \ccG \ccD_{\mres + 1}) + (\ccG^{\cci} \ccH^{\ccj} : \cci + \ccj = \mres + 2,\ 2 \le \ccj \le \mres + 2).
\]
Then \(\idealB \subseteq \idealA^{(\mres + 1)} \idealA + \idealA^{(\mres)} \idealA^{(2)} \subseteq \idealA^{(\mres + 2)}\).
We see that
\begin{alignat*}{5}
 \ccD_2 \ccD_{\mres} &\equiv (- \myy)^{\mres + 4} && \pmod{\ccx}, \\
 \ccH^2 \ccD_{\mres} &\equiv (- \myy)^{\mres + 3} z^{2\ccc_2} && \pmod{\ccx}, \\
 \ccG \ccH \ccD_{\mres} &\equiv (- \myy)^{\mres + 2} z^{\ccc_1 + 2\ccc_2} && \pmod{\ccx}, \\
 \ccH^{\mres + 2} &\equiv (- \myy)^{\mres + 2} \ccz^{(\mres + 2) \ccc_2} && \pmod{\ccx}, \\
 \ccG^2 \ccD_{\mres} &\equiv (- \myy)^{\mres + 1} z^{2 \ccc_1 + 2 \ccc_2} && \pmod{\ccx}, \\
 \ccG \ccH^{\mres + 1} &\equiv (- \myy)^{\mres + 1} \ccz^{\ccc_1 + (\mres + 2) \ccc_2} && \pmod{\ccx}, \\
 \ccG^{\cci} \ccH^{\ccj} &\equiv (- \myy)^\ccj \ccz^{(\mres + 2 - \ccj) \ccc_1 + (\mres + 2) \ccc_2} && \pmod{\ccx} \tfor 2 \le \ccj \le \mres, \\
 \ccH \ccD_{\mres + 1} &\equiv \myy \ccz^{(\mres + 1) \ccc_1 + (\mres + 1) \ccc_2 + \gamma} && \pmod{\ccx}, \\
 \ccG \ccD_{\mres + 1} &\equiv - \ccz^{(\mres + 2) \ccc_1 + (\mres + 1) \ccc_2 + \gamma} && \pmod{\ccx}.
\end{alignat*}
Note that
\[
 \min \set{\ccc_1 + 2\ccc_2, (\mres + 2) \ccc_2} = (\mres + 2) \ccc_2 - \gamma
\]
and
\[
 \min \set{2 \ccc_1 + 2 \ccc_2, \ccc_1 + (\mres + 2) \ccc_2} = \ccc_1 + (\mres + 2) \ccc_2 - \gamma.
\]
Hence
\begin{align*}
 \frac{\length{\ccR}(\ccR/\idealB + (\ccx))}{\ccb} &= (\mres + 4) 2 \ccc_2 + (\mres + 3) (\mres \ccc_2 - \gamma) \\
 &\phantom{=} + (\mres + 2) \ccc_1 + (\mres + 1)(\ccc_1 + \gamma) + \Bigl(\sum_{\ccj = 3}^{\mres} \ccj \ccc_1  \Bigr) + 2 (\ccc_1 - \ccc_2 + \gamma) + \ccc_1 \\
 &= (2 (\mres + 4) + \mres (\mres + 3) - 2) \ccc_2 + \binom{\mres + 3}{2} \ccc_1 \\
 &= \binom{\mres + 3}{2} (\ccc_1 + 2 \ccc_2).
\end{align*}
Therefore \(\idealB = \idealA^{(\mres + 2)}\).

\comment{connect: case 2}
\label{sec:org40c83ae}
\paragraph*{Case 2 $(\mres - 1) \cca_1 = \cca_2$.}

By Lemma \ref{org2a0af15},
\[
 \ccD_{\ccl} \equiv (- \ccY)^{\ccl + 1} \pmod{\ccz} \tfor 1 \le \ccl \le \mres.
\]

\comment{connect: (i)}
\label{sec:org0314c23}
\paragraph*{\textup{(i)}}
Let 
\[
 \idealB = (\ccD_{\ccl}) + (\ccG^{\cci} \ccH^{\ccj} : \cci + \ccj = \ccl,\ 0 \le \cci \le \ccl).
\]
We show that \(\length{\ccR}(\ccR/\idealB + (\ccz)) = \binom{\ccl + 1}{2}  (2 \cca_1 + \cca_2) \ccb = \binom{\ccl + 1}{2} (\mres + 1) \cca_1 \ccb\).
Since
\begin{alignat*}{2}
\ccD_{\ccl} & \equiv  (- \myy)^{\ccl + 1} & & \pmod{\ccz}, \\
\ccG^{\cci} \ccH^{\ccj} & \equiv  (- \myy)^{\cci} \ccx^{(\ccl \mres - (\mres - 1) \cci) \cca_1} & & \pmod{\ccz} \tfor 0 \le \cci \le \ccl,
\end{alignat*}
it follows that
\begin{align*}
 \frac{\length{\ccR}(\ccR/\idealB + (\ccz))}{\cca_1 \ccb} &= (\ccl + 1) \ccl + \sum_{\cci = 1}^{\ccl} \cci (\mres - 1)  \\
 &= \binom{\ccl + 1}{2} (\mres + 1).
\end{align*}
Therefore \(\idealB = \idealA^{(\ccl)}\).

\comment{connect: (ii)}
\label{sec:org000fcaa}
\paragraph*{\textup{(ii)}}
Let
\[
 \idealB = (\ccF \ccD_{\mres}, \ccH \ccD_{\mres}, \ccG \ccD_{\mres}) + (\ccG^{\cci} \ccH^{\ccj} : \cci + \ccj = \mres + 1,\ 0 \le \cci \le \mres).
\]
Since
\begin{alignat*}{3}
 \ccF \ccD_{\mres} &\equiv (- \myy)^{\mres + 3} && \pmod{\ccz},\\
 \ccG \ccD_{\mres} &\equiv (- \myy)^{\mres + 2} \ccx^{\cca_1} && \pmod{\ccz},\\
 \ccH \ccD_{\mres} &\equiv (- \myy)^{\mres + 1} \ccx^{\mres \cca_1} && \pmod{\ccz},\\
 \ccG^{\cci} \ccH^{\ccj} &\equiv (- \myy)^\cci \ccx^{(\mres(\mres + 1) -  (\mres - 1) \cci) \cca_1} && \pmod{\ccz} \tfor 0 \le \cci \le \mres,
\end{alignat*}
it follows that
\begin{align*}
 \frac{\length{\ccR}(\ccR/\idealB + (\ccz))}{\cca_1 \ccb} &=  (\mres + 3) + (\mres + 2) (\mres - 1) + (\mres + 1) \mres  + \sum_{\cci = 1}^{\mres} \cci (\mres - 1) \\
 &= 2 \mres^2 + 3 \mres + 1 + (\mres - 1) \frac{(\mres + 1) \mres}{2} \\
 &= (\mres + 1) \Bigl( 2 \mres + 1 + \frac{\mres (\mres - 1)}{2} \Bigr) \\
 &= (\mres + 1) \binom{\mres + 2}{2}.
\end{align*}
Therefore \(\idealB = \idealA^{(\mres + 1)}\).

\comment{connect: (iii)}
\label{sec:org219f8f9}
\paragraph*{\textup{(iii)}}
Let
\[
 \idealB = (\ccD_{2} \ccD_{\mres}, \ccG^2 \ccD_{\mres}, \ccG \ccH \ccD_{\mres}, \ccH^2 \ccD_{\mres}) + (\ccG^{\cci} \ccH^{\ccj} : \cci + \ccj = \mres + 2,\ 0 \le \cci \le \mres).
\]
Since
\begin{alignat*}{3}
 \ccD_2 \ccD_{\mres} &\equiv (- \myy)^{\mres + 4} && \pmod{\ccz},\\
 \ccG^2 \ccD_{\mres} &\equiv (- \myy)^{\mres + 3} \ccx^{2\cca_1} && \pmod{\ccz},\\
 \ccG \ccH \ccD_{\mres} &\equiv (- \myy)^{\mres + 2} \ccx^{(\mres + 1) \cca_1} && \pmod{\ccz}, \\
 \ccH^2 \ccD_{\mres} &\equiv (- \myy)^{\mres + 1} \ccx^{2 \mres \cca_1} && \pmod{\ccz}, \\
 \ccG^{\cci} \ccH^{\ccj} &\equiv (- \myy)^{\cci} \ccx^{(\mres(\mres + 2) - (\mres - 1) \cci) \cca_1} && \pmod{\ccz} \tfor 0 \le \cci \le \mres,
\end{alignat*}
it follows that
\begin{align*}
 \frac{\length{\ccR}(\ccR/\idealB + (\ccz))}{\cca_1 \ccb} &=  (\mres + 4) \cdot 2 + (\mres + 3) (\mres - 1) + (\mres + 2) (\mres - 1)  + (\mres + 1) \mres + \sum_{\cci = 1}^{\mres} \cci (\mres - 1) \\
 &= 2 \mres + 8 + \mres^2 + 2 \mres - 3 + \mres^2 + \mres - 2 + \mres^2 + \mres + (\mres - 1) \frac{(\mres + 1) \mres}{2}\\
 &= 3 \mres^2 + 6 \mres + 3 + (\mres - 1) \frac{(\mres + 1) \mres}{2}\\
 &= (\mres + 1) \Bigl(3 \mres + 3 + \frac{\mres^2 - \mres}{2} \Bigr)\\
 &= (\mres + 1)\binom{\mres + 3}{2}.
\end{align*}
Therefore \(\idealB = \idealA^{(\mres + 2)}\).

\qed

\section{The stable Harbourne conjecture}
\label{sec:org443784f}
\label{org87a2868}
To prove Theorem \ref{org204a6ab}, we show the following two lemmas.

\begin{lemma}
 \comment{Lem.}
\label{sec:org460e56e}
\label{org3d95cf3}
Let \(\ccM\) be a matrix of type \(1'\). For \(1 \le \ccl \le \mres[\ccM] + 1\),
\begin{equation}
\label{equ:lem-exp-dn}
\begin{split}
 \ccx^{- \alpha_{\ccl}} \ccz^{- \gamma_{\ccl}} \ccD_{\ccl} &= \ccz^{- \ccc_2} (\ccH \ccD_{\ccl - 1} - \ccx^{\cca_2 - (\ccl - 1) \cca_1} \ccG^{\ccl}) \\
 &= \ccx^{- \cca_1} (\ccG \ccD_{\ccl - 1} - \ccz^{\ccc_1 - (\ccl - 1) \ccc_2} \ccH^{\ccl}). \\
\end{split}
\end{equation}

Moreover, \(\ccx^{(\ccl - 1) \cca_1} \ccD_{\ccl} \in \idealA^{\ccl}\) and \(\ccz^{(\ccl - 1) \ccc_2} \ccD_{\ccl} \in \idealA^{\ccl}\).
 
\end{lemma}

\begin{proof}
 \comment{Proof.}
\label{sec:orgad57ab8}
We show the lemma by induction on \(\ccl\).
Note that \(\alpha_{\ccl} = \gamma_{\ccl} = 0\) for \(\ccl \le \mres[\ccM]\).
If \(\ccl = 1\), then \(\ccD_{1} = \ccF \in \idealA\) and
\[
 \ccx^{\cca_1} \ccF = - \myy \ccG - \ccz^{\ccc_1} \ccH,
\]
so \eqref{equ:lem-exp-dn} holds.
Suppose that \(\ccl \ge 2\). Since
\[
 \ccx^{- \alpha_\ccl} \ccz^{\ccc_2 - \gamma_\ccl} \ccD_{\ccl} = \ccH \ccD_{\ccl - 1} - \ccx^{\cca_2 - (\ccl - 1) \cca_1} \ccG^{\ccl},
\]
it suffices to show that
\begin{equation}
\label{equ:lem-exp-dn-goal}
 \ccx^{\cca_1} \bigl(\ccH \ccD_{\ccl - 1} - \ccx^{\cca_2 - (\ccl - 1) \cca_1} \ccG^{\ccl} \bigr) = \ccz^{\ccc_2} \bigl(\ccG \ccD_{\ccl - 1} - \ccz^{\ccc_1 - (\ccl - 1) \ccc_2} \ccH^{\ccl} \bigr).
\end{equation}
Since \(\alpha_{\ccl - 1} = \gamma_{\ccl - 1} = 0\), it follows from the induction hypothesis that
\begin{align*}
 \ccD_{\ccl - 1} &= \ccz^{-\ccc_2} \bigl (\ccH \ccD_{\ccl - 2} - \ccx^{\cca_2 - (\ccl - 2) \cca_1} \ccG^{\ccl-1} \bigr) \\
    &= \ccx^{-\cca_1} \bigl (\ccG \ccD_{\ccl - 2} - \ccz^{\ccc_1 - (\ccl - 2) \ccc_2} \ccH^{\ccl- 1} \bigr).
\end{align*}
This implies that
\begin{align*}
 \ccx^{\cca_1} \bigl(\ccH \ccD_{\ccl - 1} - \ccx^{\cca_2 - (\ccl - 1) \cca_1} \ccG^{\ccl} \bigr) &= \ccH \bigl (\ccG \ccD_{\ccl - 2} - \ccz^{\ccc_1 - (\ccl - 2) \ccc_2} \ccH^{\ccl - 1}\bigr) - \ccx^{\cca_2 - (\ccl - 2) \cca_1} \ccG^{\ccl} \\
 &= \ccG \ccH \ccD_{\ccl - 2} - \ccz^{\ccc_1 - (\ccl - 2) \ccc_2} \ccH^{\ccl} - \ccx^{\cca_2 - (\ccl - 2) \cca_1} \ccG^{\ccl}
\end{align*}
and
\begin{align*}
 \ccz^{\ccc_2} \bigl(\ccG \ccD_{\ccl - 1} - \ccz^{\ccc_1 - (\ccl - 1) \ccc_2} \ccH^{\ccl} \bigr) &= \ccG \bigl (\ccH \ccD_{\ccl - 2} - \ccx^{\cca_2 - (\ccl - 2) \cca_1} \ccG^{\ccl-1} \bigr) - \ccz^{\ccc_1 - (\ccl - 2) \ccc_2} \ccH^{\ccl} \\
 &= \ccG \ccH \ccD_{\ccl - 2} - \ccz^{\ccc_1 - (\ccl - 2) \ccc_2} \ccH^{\ccl} - \ccx^{\cca_2 - (\ccl - 2) \cca_1} \ccG^{\ccl}.
\end{align*}
Therefore \eqref{equ:lem-exp-dn-goal} and \eqref{equ:lem-exp-dn} hold.
By the induction hypothesis, 
\(\ccx^{(\ccl - 2) \cca_1} \ccD_{\ccl - 1}  \in \idealA^{\ccl - 1}\) and \(\ccz^{(\ccl - 2) \ccc_2} \ccD_{\ccl - 1} \in \idealA^{\ccl - 1}\),
so 
\begin{align*}
 \ccx^{(\ccl - 1) \cca_1} \ccD_{\ccl} &= \ccx^{(\ccl - 1) \cca_1 + \alpha_{\ccl}} \ccz^{\gamma_{\ccl}} \bigl(\ccx^{- \cca_1} (\ccG \ccD_{\ccl - 1} - \ccz^{\ccc_1 - (\ccl - 1) \ccc_2} \ccH^{\ccl})\bigr) \\
 &= \ccz^{\gamma_{\ccl}} \ccG \ccx^{(\ccl - 2) \cca_1 + \alpha_{\ccl}}  \ccD_{\ccl - 1} - \ccx^{(\ccl - 2) \cca_1 + \alpha_{\ccl}} \ccz^{\ccc_1 - (\ccl - 1) \ccc_2 + \gamma_{\ccl}} \ccH^{\ccl} \in \idealA^{\ccl}
\end{align*}
and
\begin{align*}
 \ccz^{(\ccl - 1) \ccc_2} \ccD_{\ccl} &= \ccx^{\alpha_{\ccl}} \ccz^{(\ccl - 1) \ccc_2 + \gamma_{\ccl}} \bigl(\ccz^{- \ccc_2} (\ccH \ccD_{\ccl - 1} - \ccx^{\cca_2 - (\ccl - 1) \cca_1} \ccG^{\ccl})\bigr) \\
 &= \ccx^{\alpha_{\ccl}} \ccH \ccz^{(\ccl - 2) \ccc_2 + \gamma_{\ccl}} \ccD_{\ccl - 1} - \ccx^{\cca_2 - (\ccl - 1) \cca_1 + \alpha_{\ccl}} \ccz^{(\ccl - 2) \ccc_2 + \gamma_{\ccl}} \ccG^{\ccl} \in \idealA^{\ccl}
\end{align*}
\end{proof}

\begin{lemma}
 \comment{Lem.}
\label{sec:org51b5e29}
\label{org9a2d553}
If \(\ccM\) is a matrix of type \(1'\) and \(0 \le \ccl \le \mres[\ccM] + 1\), then
\begin{equation}
\label{equ:lem-containment}
 \ccD_{\ccl} \in \clm^{\delta_{\ccl}} \idealA^{\lfloor \frac{\ccl + 1}{2} \rfloor},
\end{equation}
where
\[
 \delta_{\ccl} = \begin{cases}
 1 & \tif \ccl \text{\ is even or\ } \ccl = \mres[\ccM] + 1, \\
 0 & \totherwise.
 \end{cases}
\]
 
\end{lemma}

\begin{proof}
 \comment{Proof.}
\label{sec:orgf1f7343}
Write \(\ccl = 2 \ccn - \epsilon\) with \(\epsilon \in \set{0, 1}\). 
Then
\[
 \left\lfloor \frac{\ccl + 1}{2} \right\rfloor = \left\lfloor \frac{2\ccn - \epsilon + 1}{2} \right\rfloor = \ccn.
\]
We show that \(\ccD_{\ccl} = \ccD_{2 \ccn - \epsilon} \in \clm^{\delta_{\ccl}} \idealA^{\ccn}\) by induction on \(\ccl\).
Note that \(\mres[\ccM] = \lfloor \cca_2 / \cca_1\rfloor + 1 \ge 2\).
If \(\ccl = 1\), then \(\ccn = 1\) and \(\ccD_1 = \ccF \in \idealA\).
If \(\ccl = 2\), then \(\ccn = 1\) and \(\ccD_2 \in \clm \idealA\) by Lemma \ref{orgce83638}.
Suppose that \(\ccl = 3\); then \(\ccn = 2\).
By Remark \ref{orgbdcb8f2}, \(\ccD_{3} \in \idealA^2\).
If \(\ccl = 3 = \mres[\ccM] + 1\), then \(\mres[\ccM] = 2\) and \(\cca_2 < 2 \cca_1\),
so \(\alpha_3 = \max \set{0, 2 \cca_1 - \cca_2} > 0\) and \(\ccD_3 \in \clm \idealA^2\) by Theorem \ref{orgeb6e5b6}.
Suppose that \(\ccl \ge 4\), and let \(\alpha = \alpha_{\ccl}\), \(\gamma = \gamma_{\ccl}\), and \(\delta = \delta_{\ccl}\).
Note that \(\ccl \ge \ccn + 2\) since \(\ccl - \ccn - 2 = \ccn - \epsilon - 2 \ge 0\).

\comment{connect Step 1}
\label{sec:org40673ac}
\paragraph*{Step 1.}
We show that
\begin{equation}
\label{equ:lem-containment-y2}
 (\ccx^{\cca_2} \ccz^{\ccc_1} + \myy^2) \ccx^{\alpha} \ccD_{\ccl - 2} \in \clm^{\delta} \idealA^{\ccn}.
\end{equation}
We can see that 
\begin{equation}
\label{equ:lem-containment-y2-2}
 \ccx^{\alpha} \ccD_{\ccl - 2} \in \clm^{\delta} \idealA^{\ccn - 1}.
\end{equation}
Indeed, if \(\ccl \le \mres[\ccM]\), then \(\delta = \delta_{\ccl} = \delta_{\ccl - 2}\) and \(\ccD_{\ccl - 2} \in \clm^{\delta} \idealA^{\ccn - 1}\) by the induction hypothesis, 
so \(\ccx^{\alpha} \ccD_{\ccl - 2} \in \clm^{\delta} \idealA^{\ccn - 1}\);
if \(\ccl = \mres[\ccM] + 1\), then \(\alpha \ge 1\), so \(\ccx^\alpha \ccD_{\ccl - 2} \in \clm \idealA^{\ccn - 1} = \clm^{\delta} \idealA^{\ccn - 1}\).
Therefore \eqref{equ:lem-containment-y2-2} holds.
Hence
\[
 \ccF \ccx^{\alpha} \ccD_{\ccl -  2} = (\myy^2 - \ccx^{\cca_2} \ccz^{\ccc_1}) \ccx^{\alpha} \ccD_{\ccl - 2} \in \clm^{\delta} \idealA^{\ccn}.
\]
By Lemma \ref{org3d95cf3}, \(\ccx^{(\ccl - 3) \cca_1} \ccD_{\ccl - 2} \in \idealA^{\ccl - 2}\).
Since \(\cca_2 \ge (\mres[\ccM] - 1) \cca_1 > (\ccl - 3) \cca_1\), it follows that
\[
 \ccx^{\cca_2} \ccz^{\ccc_1} \ccD_{\ccl - 2} \in \clm \idealA^{\ccl - 2} \subseteq \clm \idealA^{\ccn}.
\]
Therefore \eqref{equ:lem-containment-y2} holds.

\comment{connect Step 2}
\label{sec:orgfb5432b}
\paragraph*{Step 2.}
Let 
\[
 \ccC_0 = \ccx^{\alpha} \bigl(x^{\cca_1 + \cca_2} \ccD_{\ccl - 2} - x^{\cca_2 - (\ccl - 2) \cca_1} \ccG^{\ccl - 1} \bigr).
\]
We show that if \(\ccC_0 \in \ccz^{\ccc_2 - \gamma} \idealA^{\ccn}\), then \(\ccD_{\ccl} \in \clm^{\delta} \idealA^{\ccn}\).
By Lemma \ref{org3d95cf3},
\begin{align*}
 \ccx^{- \alpha} \ccz^{\ccc_2 - \gamma} \ccD_{\ccl} &= \ccH \ccD_{\ccl - 1} - \ccx^{\cca_2 - (\ccl - 1) \cca_1} \ccG^{\ccl} \\
 &= (\ccx^{\cca_1 + \cca_2} - \myy \ccz^{\ccc_2}) \ccD_{\ccl - 1} - \ccx^{\cca_2 - (\ccl - 1) \cca_1} \ccG^{\ccl} \\
 &= \ccx^{\cca_2} \bigl(\ccG \ccD_{\ccl - 2} - \ccz^{\ccc_1 - (\ccl - 2) \ccc_2} \ccH^{\ccl - 1} \bigr) \\
 &\phantom{=} - \myy \bigl (\ccH \ccD_{\ccl - 2} - \ccx^{\cca_2 - (\ccl - 2) \cca_1} \ccG^{\ccl - 1} \bigr) \\
 &\phantom{=} - \ccx^{\cca_2 - (\ccl - 1) \cca_1} \bigl(\ccz^{\ccc_1 + \ccc_2} - \myy \ccx^{\cca_1} \bigr) \ccG^{\ccl - 1} \\
 &= (\ccx^{\cca_2} \ccG - \myy \ccH) \ccD_{\ccl - 2} + 2 \myy \ccx^{\cca_2 - (\ccl - 2) \cca_1} \ccG^{\ccl - 1}  \\
 &\phantom{=} - \ccx^{\cca_2} \ccz^{\ccc_1 - (\ccl - 2) \ccc_2} \ccH^{\ccl - 1} - \ccx^{\cca_2 - (\ccl - 1) \cca_1} \ccz^{\ccc_1 + \ccc_2} \ccG^{\ccl - 1}.
\end{align*}
Notice that \(\ccc_1 - (\ccl - 2) \ccc_2 \ge \ccc_2 - \gamma\) since
\[
 (\ccc_1 - (\ccl - 2) \ccc_2) - (\ccc_2 - \gamma) \ge  \ccc_1 - (\ccl - 2) \ccc_2 - \ccc_2 + (\ccl - 1) \ccc_2 - \ccc_1 = 0.
\]
Thus
\[
 \ccz^{\ccc_2 - \gamma} \ccD_{\ccl} \equiv \ccx^{\alpha} \bigl((\ccx^{\cca_2} \ccG - \myy \ccH) \ccD_{\ccl - 2} + 2 \myy \ccx^{\cca_2 - (\ccl - 2) \cca_1} \ccG^{\ccl - 1} \bigr) \pmod {\ccz^{\ccc_2 - \gamma} \clm^{\delta} \idealA^{\ccn}}.
\]
Since
\begin{align*}
 \ccx^{\cca_2} \ccG - \myy \ccH &= \ccx^{\cca_2} \bigl(\ccz^{\ccc_1 + \ccc_2} - \myy \ccx^{\cca_1}) - \myy (\ccx^{\cca_1 + \cca_2} - \myy \ccz^{\ccc_2} \bigr) \\
 &= (\ccx^{\cca_2} \ccz^{\ccc_1} + \myy^2) \ccz^{\ccc_2} - 2 \myy \ccx^{\cca_1 + \cca_2},
\end{align*}
it follows from \eqref{equ:lem-containment-y2} that
\begin{align*}
 \ccz^{\ccc_2 - \gamma} \ccD_{\ccl} &\equiv \ccx^{\alpha} \bigl((\ccx^{\cca_2} \ccz^{\ccc_1} + \myy^2) \ccz^{\ccc_2} \ccD_{\ccl - 2} - 2 \myy \ccx^{\cca_1 + \cca_2} \ccD_{\ccl - 2} + 2 \myy \ccx^{\cca_2 - (\ccl - 2) \cca_1} \ccG^{\ccl - 1} \bigr) \\
 &\equiv \ccx^{\alpha} \bigl(- 2 \myy \ccx^{\cca_1 + \cca_2} \ccD_{\ccl - 2} + 2 \myy \ccx^{\cca_2 - (\ccl - 2) \cca_1} \ccG^{\ccl - 1} \bigr) \\
 &\equiv -2 \myy \ccC_0 \pmod{\ccz^{\ccc_2 - \gamma} \clm^{\delta} \idealA^{\ccn}}.
\end{align*}
Thus if \(\ccC_0 \in \ccz^{\ccc_2 - \gamma} \idealA^{\ccn}\), then \(\ccz^{\ccc_2 - \gamma} \ccD_{\ccl} \in \ccz^{\ccc_2 - \gamma}  \clm^{\delta} \idealA^{\ccn}\), 
so \(\ccD_{\ccl} \in \clm^{\delta} \idealA^{\ccn}\).

\comment{connect Step 3}
\label{sec:orgea5da72}
\paragraph*{Step 3.}
For \(1 \le \ccj \le \ccn\), let
\[
 \ccC_{\ccj} = \ccx^{\alpha}\bigl(x^{\cca_2 - (\ccj - 1) \cca_1} \ccD_{\ccl - \ccj - 2} - x^{\cca_2 - (\ccl - 2) \cca_1} \ccG^{\ccl - \ccj - 1} \bigr).
\]
By induction on \(\ccj\), we show that 
if \(\ccC_{\ccj} \in \ccz^{\ccc_2 - \gamma} \idealA^{\ccn - \ccj}\), then \(\ccC_0 \in \ccz^{\ccc_2 - \gamma} \idealA^{\ccn}\).
If \(\ccj = 0\), then it is obvious.
Suppose that \(\ccj \ge 1\). By Lemma \ref{org3d95cf3},
\begin{align*}
 \ccC_{\ccj - 1} &= \ccx^{\alpha} \bigl(x^{\cca_2 - (\ccj - 2) \cca_1} \ccD_{\ccl - \ccj - 1} - x^{\cca_2 - (\ccl - 2) \cca_1} \ccG^{\ccl - \ccj} \bigr) \\
 &= \ccx^{\alpha} \bigl(x^{\cca_2 - (\ccj - 1) \cca_1} (\ccG \ccD_{\ccl - \ccj - 2} - \ccz^{\ccc_1 - (\ccl - \ccj - 2) \ccc_2} \ccH^{\ccl - \ccj - 1}) - x^{\cca_2 - (\ccl - 2) \cca_1} \ccG^{\ccl - \ccj} \bigr).
\end{align*}
Since \(\ccc_1 - (\ccl - \ccj - 2) \ccc_2 \ge \ccc_1 - (\ccl - 2) \ccc_2 \ge \ccc_2 - \gamma\), it follows that
\begin{align*}
 \ccC_{\ccj - 1} &\equiv \ccx^{\alpha} \ccG(x^{\cca_2 - (\ccj - 1) \cca_1} \ccD_{\ccl - \ccj - 2} - x^{\cca_2 - (\ccl - 2) \cca_1} \ccG^{\ccl - \ccj - 1 })  \\
  &\equiv \ccG \ccC_{\ccj} \pmod {\ccz^{\ccc_2 - \gamma} \idealA^{\ccn - \ccj + 1}}.
\end{align*}
Therefore if \(\ccC_{\ccj} \in \ccz^{\ccc_2 - \gamma} \idealA^{\ccn - \ccj}\), then \(\ccC_{j - 1} \in \ccz^{\ccc_2 - \gamma} \idealA^{\ccn - \ccj + 1}\),
so \(\ccC_0 \in \ccz^{\ccc_2 - \gamma} \idealA^{\ccn}\) by the induction hypothesis.

\comment{connect Step 4}
\label{sec:org3c9a420}
\paragraph*{Step 4.}

We show that \(\ccC_{\ccn} \in (\ccz^{\ccc_2 - \gamma})\).
Note that
\begin{align*}
 \ccC_{\ccn} &= \ccx^{\alpha} \bigl(\ccx^{\cca_2 - (\ccn - 1) \cca_1} \ccD_{\ccl - \ccn - 2} - \ccx^{\cca_2 - (\ccl - 2) \cca_1} \ccG^{\ccl - \ccn - 1}\bigr) \\
 &= \ccx^{\alpha} \bigl(\ccx^{\cca_2 - (\ccn - 1) \cca_1} \ccD_{\ccn - \epsilon - 2} - \ccx^{\cca_2 - (2 \ccn - \epsilon - 2) \cca_1} \ccG^{\ccn - \epsilon - 1} \bigr).
\end{align*}
By Lemma \ref{org2a0af15},
\[
 \ccD_{\ccn - \epsilon - 2} \equiv (- \myy)^{\ccn - \epsilon - 1} \pmod{\ccz^{\ccc_1 - (\ccn - \epsilon - 3) \ccc_2}}.
\]
We see that \(\ccc_1 - (\ccn - \epsilon - 3) \ccc_2 \ge \ccc_2 - \gamma\), so
\[
 \ccD_{\ccn - \epsilon - 2} \equiv (- \myy)^{\ccn - \epsilon - 1} \pmod{\ccz^{\ccc_2 - \gamma}}.
\]
Since
\begin{align*}
 \ccG^{\ccn - \epsilon - 1} &= (\ccz^{\ccc_1 + \ccc_2} - \myy \ccx^{\cca_1})^{\ccn - \epsilon - 1} \\
 &\equiv (- \myy \ccx^{\cca_1})^{\ccn - \epsilon - 1} \pmod{\ccz^{\ccc_2 - \gamma}},
\end{align*}
it follows that
\begin{align*}
 \ccC_{\ccn} &= \ccx^{\alpha} \bigl(\ccx^{\cca_2 - (\ccn - 1) \cca_1} \ccD_{\ccn - \epsilon - 2} - \ccx^{\cca_2 - (2 \ccn - \epsilon - 2) \cca_1} \ccG^{\ccn - \epsilon - 1} \bigr) \\
 &\equiv \ccx^{\alpha} \bigl(\ccx^{\cca_2 - (\ccn - 1) \cca_1} (- \myy)^{\ccn - \epsilon - 1}  - \ccx^{\cca_2 - (2 \ccn - \epsilon - 2) \cca_1} (- \myy \ccx^{\cca_1})^{\ccn - \epsilon - 1} \bigr) \\
 &\equiv 0 \pmod{\ccz^{\ccc_2 - \gamma}}.
\end{align*}
Therefore \(\ccC_{\ccn} \in (\ccz^{\ccc_2 - \gamma})\). This completes the proof.
\end{proof}

\comment{connect}
\label{sec:org099269b}

\vspace{0.5em}

\noindent
\emph{Proof of Theorem~\ref{org204a6ab}}. 
We first show that \(\idealA^{(2 \ccl - 1)} \not \subseteq \clm \idealA^{\ccl}\) for \(1 \le \ccl \le \ccn - 1\).
Since \(\ccn = \lfloor (\mres[\ccM] + 1)/2 \rfloor + 1 \le (\mres[\ccM] + 1)/2 + 1\),
we see that \(2 \ccl - 1 \le 2 \ccn - 3 \le \mres[\ccM]\).
By Lemma \ref{org2a0af15}, \(\ccD_{2\ccl - 1} \equiv (-\myy)^{2\ccl} \pmod{\ccx}\) or \(\ccD_{2\ccl - 1} \equiv (-\myy)^{2\ccl} \pmod{\ccz}\),
so \(\ccD_{2 \ccl - 1} \not \in \clm \idealA^{\ccl}\). Therefore \(\idealA^{(2 \ccl - 1)} \not \subseteq \clm \idealA^{\ccl}\) by Theorem \ref{org319617b}.

We next show that \(\idealA^{(2 \ccn - 1)} \subseteq \clm \idealA^{\ccn}\). 
First, suppose that \(\mres[\ccM]\) is odd; then \(\mres[\ccM] = 2 \ccn - 3\).
Note that \(\ccn \ge 2\).
By Theorem~\ref{org319617b},
\begin{align*}
 \idealA^{(2 \ccn - 3)} &= \idealA^{2 \ccn - 3} + (\ccD_{2 \ccn - 3}),\\
 \idealA^{(2 \ccn - 2)} &= \idealA^{(2 \ccn - 3)} \idealA + (\ccD_{2 \ccn - 2}),\\
 \idealA^{(2 \ccn - 1)} &= \idealA^{(2 \ccn - 2)} \idealA + \idealA^{(2 \ccn - 3)} \idealA^{(2)}.
\end{align*}
It follows from Lemmas \ref{orgce83638} and \ref{org9a2d553} that \(\idealA^{(2)} \subseteq \clm \idealA\),
\begin{align*}
 \idealA^{(2 \ccn - 3)} &= \idealA^{2 \ccn - 3} + (\ccD_{2 \ccn - 3}) 
 \subseteq \idealA^{2 \ccn - 3} + \idealA^{\ccn - 1} = \idealA^{\ccn - 1},
\end{align*}
and
\begin{align*}
 \idealA^{(2 \ccn - 2)} &= \idealA^{(2 \ccn - 3)} \idealA + (\ccD_{2 \ccn - 2}) 
 \subseteq \idealA^{\ccn} + \clm \idealA^{\ccn - 1} = \clm \idealA^{\ccn - 1}.
\end{align*}
Therefore
\begin{align*}
 \idealA^{(2 \ccn - 1)} &= \idealA^{(2 \ccn - 2)} \idealA + \idealA^{(2 \ccn - 3)} \idealA^{(2)} 
 \subseteq \clm \idealA^{\ccn - 1} \idealA + \idealA^{\ccn - 1} \clm \idealA = \clm \idealA^{\ccn}.
\end{align*}
Next, suppose that \(\mres[\ccM]\) is even; then \(\mres[\ccM] = 2 \ccn - 2\).
By Theorem~\ref{org319617b},
\begin{align*}
 \idealA^{(2 \ccn - 2)} &= \idealA^{2 \ccn - 2} + (\ccD_{2 \ccn - 2}),\\
 \idealA^{(2 \ccn - 1)} &= \idealA^{(2 \ccn - 2)} \idealA + (\ccD_{2 \ccn - 1}).
\end{align*}
By Lemma \ref{org9a2d553},
\[
  \idealA^{(2 \ccn - 2)} = \idealA^{2 \ccn - 2} + (\ccD_{2 \ccn - 2}) \subseteq \idealA^{2 \ccn - 2} + \clm \idealA^{\ccn - 1} = \clm \idealA^{\ccn - 1}
\]
and
\[
  \idealA^{(2 \ccn - 1)} = \idealA^{(2 \ccn - 2)} \idealA + (\ccD_{2 \ccn - 1}) \subseteq \clm \idealA^{\ccn - 1} \idealA + \clm \idealA^{\ccn} = \clm \idealA^{\ccn}.
\]
This completes the proof.
\qed

\comment{bibliography}
\label{sec:org4596dc1}

\vspace{0.5em}

\begin{tabular}{@{}l@{}}%
  \textsc{Kosuke Fukumuro}\\
  Department of Mathematics National Institute of Technology, \\
  Kisarazu College, Chiba, Japan\\
  \texttt{fukumuro@n.kisarazu.ac.jp, blackbox@tempo.ocn.ne.jp}
\end{tabular}

\vspace{1em}

\begin{tabular}{@{}l@{}}%
  \textsc{Yuki Irie}\\
  Research Alliance Center for Mathematical Sciences,\\
   Tohoku University, Miyagi, Japan\\
  \texttt{yirie@tohoku.ac.jp, yuki@yirie.info}
\end{tabular}
\end{document}